\documentclass{article}
\usepackage[utf8]{inputenc}
\usepackage{todonotes}
\usepackage{caption}
\usepackage{nicefrac}
\usepackage{amsmath,enumerate}
 \usepackage{amsthm}
\usepackage{amsfonts}
\usepackage{amssymb}
\usepackage{mathrsfs}
\usepackage{pgfplots}
\usepackage{mathtools}
\usepackage{comment}
\usepackage{lmodern}
\usepackage{bm}
\usepackage{subfig}
\usepackage{caption}
\usepackage{cite}
\usepackage{microtype}
\usepackage[margin=1.75in]{geometry}
\usepackage{tikz}
\usetikzlibrary{shapes,arrows,graphs,graphs.standard,shapes.misc}
\usetikzlibrary{matrix,decorations.pathreplacing, calc, positioning,fit}
\usetikzlibrary{shapes.geometric}

\newtheorem{Theorem}{Theorem}
\newtheorem{Definition}{Definition}

\newtheorem{Proposition}[Theorem]{Proposition}
\newtheorem{Lemma}{Lemma}
\newtheorem{Corollary}[Theorem]{Corollary}

\newtheorem{Remark}{Remark}

\newcommand{\R}{\mathbb{R}}

\newcommand{\rint}{\operatorname{int}}

\DeclareMathOperator{\diag}{diag}

\newcommand{\bfo}{{\bf 1}}

\newcommand{\cH}{\mathcal{H}}
\newcommand{\cA}{\mathcal{A}}
\newcommand{\cS}{\mathcal{S}}
\newcommand{\cY}{\mathcal{Y}}

\newcommand{\supp}{\operatorname{supp}}

\title{Gossip over Holonomic Graphs}

\author{Xudong Chen\footnote{ECEE Department, CU Boulder, email: \texttt{xudong.chen@colorado.edu}.}, \quad Mohamed-Ali Belabbas\footnote{ECE Department and Coordinated Science Laboratory, email: \texttt{belabbas@illinois.edu}.}, \quad Ji Liu\footnote{ECE Department, Stony Brook University, email: \texttt{ji.liu@stonybrook.edu}.}}

\begin{document}

\date{}
\maketitle 

\begin{abstract}
    A gossip process is an iterative process in a multi-agent system  where only two neighboring agents communicate at each iteration and update their states. The neighboring condition is by convention described by an undirected graph. In this paper, we consider a general update rule whereby each agent  takes an arbitrary weighted average of its and its neighbor's current states. In general, the limit of the gossip process (if it converges) depends on the order of iterations of the gossiping pairs. The main contribution of the paper is to provide a necessary and sufficient condition for convergence of the gossip process that is independent of the order of iterations. This result relies on the introduction of the novel notion of holonomy of local stochastic matrices for the communication graph. We also provide complete characterizations of the limit and the space of holonomic stochastic matrices over the graph.
    \vspace{.3cm}


\end{abstract}

\section{Introduction}

Consensus problems have a long history~\cite{degroot} and are closely related to Markov chains~\cite{seneta}.
Over the past decades, there has been considerable interest in developing algorithms intended to cause a group of $n>1$ agents to reach a consensus in a distributed manner, see~\cite{Ts3,vicsekmodel,reza1,luc,ReBe05,reachingp1,cutcdc,touri2014,decide,tacrate,xudong} just to cite a few. 
A simple idea to solve the consensus problem exploits a form of iterative message passing,  in which each agent  exchanges information with at most {\em one} other agent per iteration. One such exchange is called a {\em gossip}. Whenever two agents gossip, they set their state variables equal to the average of their values before gossiping~\cite{boyd052}. This process, which we term {\em standard gossip process}, is known to make all agents' values converge to the average of their initial states, provided that the neighbor graph is connected. 
In the standard gossip process, the update matrix associated with each iteration is {\em doubly} stochastic. 
Recall that all doubly stochastic matrices share the same left- (and right-) eigenvector, namely the vector $\bfo$ of all $1$'s, corresponding to eigenvalue $1$. Therefore, if the product of doubly stochastic matrices converge to a rank-one matrix, it can only converge to $\frac{1}{n}\bfo \bfo^\top$, where $n$ is the number of agents in the system.

In this paper, we enable convergence of a gossip process to an arbitrary {\em weighted} ensemble average. To this end, the update rule of the standard gossip process is generalized to allow neighboring agents to update their states according to a {\em weighted} average of their current values. 
We emphasize that when a pair of agents, say $i$ and $j$, communicate, they are not required to take the same weighted average. For example, agent $i$ can weigh its and $j$'s values by $1/3$ and $2/3$, whereas agent $j$'s weighs its and $i$'s values by $2/5$ and $3/5$.
We call this generalized version of a gossip process a {\em weighted gossip} process.

The extension of the update rule to (asymmetric) weighted average at each iteration gives rise to several important questions: (1) Can we still guarantee (exponential) convergence of the weighted gossip process? (2) Since different gossiping pairs can take different weighted averages and the corresponding stochastic matrices are {\em not} necessarily doubly stochastic, can we characterize the limit of a weighted gossip process (provided that the product of those stochastic matrices converges to a rank-one matrix)? (3) Furthermore, when can the limit of the product be {\em independent} of the order of appearance of the stochastic matrices in the product?  (4) Finally, in settings for which  questions (1)-(3) have positive answers,  can we design an update rule (or, equivalently, the set of stochastic matrices) to ensure convergence of the gossip process to {\em any} desired weighted ensemble average?

We address in the paper the above four questions and provide answers to them. To do so, we first introduce a novel notion of holonomy of stochastic matrices assigned to the edges of a given undirected graph that describes the neighbor topology (also referred to as communication topology) of the multi-agent system. 
We borrow this terminology from Riemannian geometry, which encodes transformations of vectors transported along close curves. 
In the present context, we take products of stochastic matrices along edges in arbitrary walks in the graph, and the term holonomy is used to indicate a change of a certain eigenvector corresponding to eigenvalue $1$ of the product along any closed walk in the graph. Holonomy in this context  is related to the so-called {\em graph balance} for signed (or, more generally, voltage) graphs~\cite{chen2017cluster}, as both concepts require the net effect of a contextually-defined transformation over a closed walk to be trivial. 

Based on the notion of holonomy, we establish a necessary and sufficient condition for a weighted gossip process to (exponentially) converge to a {\em unique} rank-one matrix, and characterize explicitly this limit. Note that any such limit can always be written as $\bfo p^\top$ where $p$ is a probability vector. 
The above facts imply the existence of a map $\pi$ that assigns a set of holonomic stochastic matrices to the probability vector $p$.

Moreover, we show that for an arbitrary probability vector $p$ in the interior of the standard simplex, there exists a set of holonomic stochastic matrices so that a corresponding weighted gossip process converges to $\bfo p^\top$, thus providing an affirmative answer to question (4) above.  
In other words, we show that the map $\pi$ is onto the interior of the standard simplex.
Another major contribution of the paper is to provide a complete characterization of the preimage $\pi^{-1}(p)$, i.e., all sets of holonomic stochastic matrices that are mapped to $p$ by $\pi$.

This paper shares the same spirit as the recent  work~\cite{BC2020triangulated}, in the sense that both consider convergence to arbitrary rank-one stochastic matrices, and both provide conditions for the limits of products of certain stochastic matrices to be independent of their orders of appearance in the products. However, the specific settings and analyses  differ significantly. In the present work,  the  stochastic matrices considered have a nontrivial $2\times 2$ principal submatrix, with the remaining part being an identity matrix,  reflecting the fact that communication at each iteration is pairwise. Stochastic matrices appearing in~\cite{BC2020triangulated} have in contrast a non-trivial $3\times 3$ principal submatrix, reflecting communications for three agents simultaneously. This seemingly minor extension in fact increases the complexity of the products drastically. As a trade-off for the reduced complexity of the products in the current paper, we do not impose any restriction on the structure of the graph (as long as it is connected); whereas in~\cite{BC2020triangulated}, only a special class of rigid graphs, termed triangulated Laman graphs, allowed us to draw conclusions similar to the ones of the current paper.

Our work answers questions about weighted gossip processes that have not been investigated in the extant literature. For a comparison with existing works, we describe a few recent results about the {\em standard} gossip process.  
For a deterministic standard gossip process, if each pair of neighboring agents gossip infinitely often, then all agents' states asymptotically converge to the average of their initial values~\cite{cutcdc}; if there exists a period $T$ such that each pair gossip at least once within each successive subsequence of length $T$, the converge will be reached exponentially fast~\cite{pieee}. Moreover, if the underlying graph is a tree and each neighboring pair is restricted to gossip only once per period, it is known~\cite{tacgossip} that the convergence rate is fixed and invariant over all possible periodic gossip sequences the graph allows. For a randomized standard gossip process, in which each pair of neighbor agents are randomly selected to gossip, all agents' states converge to the average of their initial values almost surely and in mean square~\cite{boyd052}. 
Finally, we emphasize that the terms ``gossip'' and ``weighted gossip'' have, over the years, had evolving meanings. We defined here a gossip process as being an iterative process in which interactions are between pairs of agents only. For some, a ``gossip'' process is moreover required to have agents converge to the average of their initial states~\cite{boyd052}; also, ``weighted gossip'' is used in~\cite{wgossip} to describe a variant of the push-sum algorithm~\cite{pushsum}, whose purpose is to reach an average consensus over directed graphs. These works thus differ from ours.

The remainder of the paper is organized as follows: We gather a few key notations and conventions at the end of the section. The notion of holonomy and the main results of the paper are presented in Section~\ref{sec:mainresult}. Analyses and proofs of the main results are provided in Section~\ref{sec:proofs}. The paper ends with a conclusion in Section~\ref{end}.

\vspace{.1cm}

\noindent
{\bf Notations and conventions.} 
We denote by $G= (V, E)$ an undirected graph, without multiple edges but, possibly, with self-loops. We call $G$ {\em simple} if it  has no self-loops.  The graphs we consider here are connected.
The node set is by convention denoted by $V = \{v_1,\ldots, v_n\}$ and the edge set by $E$. We refer to the edge linking nodes $v_i$ and $v_j$ as  $(v_i,v_j)$. A self-loop is then of the form $(v_i,v_i)$.   

Given a sequence of edges  $\gamma=e_1\cdots e_k$ in $E$, we say that a node $v\in V$ is {\em covered} by $\gamma$ if it is incident to an edge in $\gamma$. For $\gamma = e_1e_2\cdots$ an infinite sequence, we say that $v$ is covered infinitely often by $\gamma$ if there exists an infinite number of  sub-indices $i_1<i_2<\cdots$ such that $e_{i_j}$ is incident to $v$. 

Given a sequence $\gamma=e_1e_2\cdots$, we say that $\gamma'$ is a string of $\gamma$ if it is a contiguous subsequence, i.e., $\gamma' = e_ie_{i+1}\cdots e_\ell$ for some $i \ge 1$ and $\ell \ge i$.  
Let $\gamma = e_1\cdots e_k$ be a finite sequence and $e_{k+1}$ be an edge of $G$.  Denote by $\gamma \vee e_{k+1}$ the sequence $e_1\cdots e_k e_{k+1}$ obtained by adding $e_{k+1}$ to the end of $\gamma$.

For a given undirected graph $G$ as above, we denote by $\vec G = (V, \vec E)$ a {\em directed graph} on the same node set and with a ``bidirectionalized'' edge set; precisely, $\vec E$ is defined as follows: we assign to every edge $(v_i,v_j)$ of $G$, $i\neq j$,  two directed edges $v_iv_j$ and $v_jv_i$; to a self-loop $(v_i,v_i)$ of $G$ corresponds a self-loop $v_iv_i$ of $\vec G$.

Let $w = v_{i_1}\ldots v_{i_k}$ be a walk in the directed graph $\vec G$, i.e., every $v_{i_\ell}v_{i_{\ell + 1}}$, for $\ell =1,\ldots, k-1$, is an edge of $\vec G$. We call $v_{i_1}$ the starting-node and $v_{i_k}$ the ending-node of~$w$. We define by $w^{-1}:= v_{i_k} v_{i_{k-1}} \ldots v_{i_1}$ the {\em inverse} of $w$. 
Let $w' = v_{i_k}v_{i_{k+1}}\cdots v_{i_m}$ be another walk in $G$, where the starting-node of $w'$ is the same as the ending-node of $w$. We denote by $ww' = v_{i_1}\ldots v_{i_k}\ldots v_{i_m}$ the {\em concatenation} of the two walks.

A square nonnegative matrix is called a {\em stochastic matrix} if all its row-sums equal one. 
A matrix is {\em irreducible} if it is not similar via a permutation to a block upper triangular matrix (with strictly more than one block of positive size). 
The graph of an $n\times n$ matrix is a directed graph on $n$ nodes: there is a directed edge from node $v_j$ to node $v_i$ whenever the $ij$th entry of the matrix is nonzero. 
It is known that a matrix is irreducible if and only if its graph is strongly connected~\cite[Theorem 6.2.24]{horn1}. 

On the space of $n \times m$ real matrices, we define the following semi-norm: for a given $A\in\R^{n\times m}$, 
$$\|A\|_S := \max_{1\le j\le m}\max_{1\le i_1,i_2\le n}|a_{i_1j}-a_{i_2j}|.$$
The zero-set of this semi-norm is the set of matrices with all rows equal. See~\cite{wolf63} for more details.

The {\em support of a matrix} $A = [a_{ij}]$, denoted by $\supp(A)$, is the set of indices $ij$ such that $a_{ij}\neq 0$. We denote by $\min A$ the smallest non-zero entry of $A$: 
$$\min A = \min_{ij \in \supp(A)} a_{ij}.$$ 
In this paper, we will only consider $\min A$ for $A$ being a nonnegative matrix.

We say that $p \in \R^n$ is a {\em probability vector} if $p_i \geq 0$ and $\sum_{i=1}^n p_i=1$. The set of probability vectors in $\R^n$ is the $(n-1)$-simplex, which is denoted by $\Delta^{n-1}$. Its interior with respect to the standard Euclidean topology in $\R^n$ is denoted by $\rint \Delta^{n-1}$. If $p\in \rint \Delta^{n-1}$, then all entries of $p$ are positive. 

We let $\bfo$ be a vector of all ones, whose  dimension will be clear from the context.  

Given a real number $x$, we denote by $\lfloor x \rfloor$  the largest integer that is smaller than or equal to $x$, i.e., $\lfloor x \rfloor := \max_{z \in \mathbb{Z}} \lbrace z \mid z \leq x \rbrace$.

The Cartesian product of $k$ linearly independent open bounded line segments in an Euclidean space is called a { \em $k$-dimensional open box}. 
An open box is not necessarily  parallel to the coordinate axes.

We denote by $\R_+$ the set of positive real numbers.

\section{Main Results}\label{sec:mainresult}
We now present the main results proved in this paper, and the main concepts introduced.

\subsection{Local stochastic matrices and holonomy for digraphs}\label{ssec:lsmholo}

Let $G = (V, E)$ be a simple graph on $n$ nodes. Each node represents an agent, and each agent is assigned a variable $x_i(t) \in \R$ at the time step~$t$. To each edge $(v_i,v_j)$ of $G$, with $i < j$, corresponds a potential interaction of agents $i$ and $j$, whereby they update their current  states $x_i(t)$ and $x_j(t)$ (if this gossip pair is activated) according to the rule:
\begin{equation*}\label{eq:defupdate}
\begin{bmatrix} x_i(t+1) \\ x_j(t+1) \end{bmatrix}=\bar A_{ij}\begin{bmatrix} x_i(t) \\ x_j(t) \end{bmatrix},
\end{equation*}
where $\bar A_{ij}$ is the $2$-by-$2$ row stochastic matrix given by
\begin{equation}\label{eq:defbarAij}
\bar A_{ij}:=  \begin{bmatrix} 1-a_{ij} & a_{ij} \\ a_{ji} & 1-a_{ji}
\end{bmatrix},
\end{equation}
with $a_{ij}$ and $a_{ji}$ real numbers in the open interval $(0,1)$. 

During this update, all the other agents $k$, for $k\neq i, j$, keep their states unchanged: $x_k(t+1)=x_k(t)$. 
Thus, 
if we let $E_{ij}$ be the $n$-by-$n$ square matrix with $1$ at the $ij$th entry and $0$ elsewhere, then the update of the entire network can be described by $x(t+1) = A_{ij} x(t)$, where $A_{ij}$ is the $n$-by-$n$ row stochastic matrix defined as follows: 
\begin{align}\label{eq:deflsm}
    A_{ij} := \; & (1-a_{ij}) E_{ii} + a_{ij} E_{ij}
    + a_{ji} E_{ji} + (1-a_{ji}) E_{jj} \nonumber\\
    & + \sum_{k\neq i, j} E_{kk}. 
\end{align}
In words, the matrix $A_{ij}$ is such that the principal submatrix associated with columns/rows $i$ and $j$ is the $2\times 2$ stochastic matrix $\bar A_{ij}$, and the complementary principle submatrix is the identity matrix $I_{n-2}$. Note that $A_{ij} = A_{ji}$ from~\eqref{eq:deflsm}.  

We call these $A_{ij}$'s, for $(v_i,v_j)\in E$, {\bf local stochastic matrices} of $G$. 
The graph of each local stochastic matrix is a bi-directional graph with exactly two directed edges $v_iv_j$ and $v_jv_i$, and self-arcs at all $n$ nodes.

A local stochastic matrix $A_{ij}$, for $(v_i,v_j)$ an edge of $G$, has two degrees of freedom, namely $a_{ij}$ and $a_{ji}$ as defined in~\eqref{eq:deflsm}. 
We denote by $\cS_G$ the set of $|E|$-tuples of local stochastic matrices over a connected graph $G = (V, E)$, which is an open convex subset of an Euclidean space of dimension $2|E|$. 
Given an ordering of the edges in $G$, we will use $\cA=(A_{ij})_{(v_i,v_j) \in E}$ to denote an element of $\cS_G$.

For a finite sequence $\gamma = e_{1} \ldots e_{k}$ of edges in $G$ and for a given pair of integers $0\le s \le t \le k$, we define a product of local stochastic matrices as follows: for $t\geq s+1$,  
$$
P_\gamma(t:s) := A_{e_t} A_{e_{t - 1}} \cdots A_{e_{s+1}},
$$
and $P_\gamma(t:s)=I$ for $t \leq s$. For the case where $s = 0$ and $t = k$, we will simply write $P_\gamma = P_\gamma(k:0)$. 
The notation can be used on infinite strings $\gamma$, with $k = \infty$, as well. We single out the following sequences: 

\vspace{.15in}

\begin{Definition}[Spanning sequence]\label{def:spanningsequence} Let $G = (V, E)$ be a simple, undirected graph. 
A finite sequence of edges of $G$ is {\bf spanning} if it covers a spanning tree of $G$.  
An infinite sequence of edges is {\em spanning} if it has infinitely many disjoint finite strings that are spanning. 
An infinite sequence is {\em $m$-spanning} if every string of length $m$ is spanning. 
\end{Definition}

Let $\vec G = (V, \vec E)$ be the directed version of $G$. 
For each directed edge $v_iv_j$ in $\vec G$, we define the ratio
$$
r_{ij} :=  \frac{a_{ij}}{a_{ji}}.
$$
Note that $r_{ij}$ is well-defined, because $a_{ji}\in (0,1)$. Also, it follows from the definition that $r_{ji}=r_{ij}^{-1}$.
Let $w = v_{i_1}\ldots v_{i_k}$ be a walk in $\vec G$. We define
\begin{equation}\label{eq:defRw}
R_w:= \prod^{k-1}_{\ell = 1} r_{i_\ell,i_{\ell+1}}.  
\end{equation}

Let $w_1$ and $w_2$ be two walks in $G$, with $w_1$ ending at the starting node of $w_2$. 
Then, $R_{w_1w_2}=R_{w_1}R_{w_2}$. In particular, setting $w := w_1=w_2^{-1}$, we get $R_{ww^{-1}}=1$.

The following definition is instrumental to our results:

\vspace{.15in}

\begin{Definition}[Holonomic local stochastic matrices]\label{def:holonomicAij}
Let $C$ be a cycle in $\vec G$ of length greater than $2$. The local stochastic matrices $A_{ij}$ are {\bf holonomic} for $C$ if $R_C=1$, and are {\it holonomic} for $G$ if they are holonomic for every cycle of $\vec G$ of length of greater than $2$.
\end{Definition}

With foresight, we borrow the word {\em holonomic} from differential geometry to characterize the set of matrices. The justification of the name is the following: in geometry, this notion,  roughly speaking, describes variation of some quantity (e.g., a parallel-transported vector) along loops in a given space. If there is no variation of the quantity while `traveling' around the loop, the process is said to be `holonomic'. Here, the space is the graph and the quantity is the product of the ratios $r_{ij}$ along cycles of the graph. The notion of holonomy will appear through a formula, established below, that involves  the products of $r_{ij}$'s along walks in $G$. 
Clearly, for the product of these $r_{ij}$ along walks to depend {\em only} on the starting- and ending-nodes, it is necessary that such products along cycles be equal to~$1$; indeed, these cycles can be inserted an arbitrary amount of times to a walk without changing its starting nor its ending node. 
Because of such one-to-one correspondence, we use the term `holonomy' as a definition of the properties of the matrices $A_{ij}$ described in Definition~\ref{def:holonomicAij}.

\vspace{.15in}

\begin{Remark}\normalfont
Note that if $C$ is a $2$-cycle, then $R_C$ is $1$ by definition. Thus, if $G$ is a tree, then every set of local stochastic matrices is holonomic for $G$.    
\end{Remark}

\subsection{Statement of the main results}
We have three main results: the first two deal with convergence of infinite products of stochastic matrices and existence of a unique limiting rank one matrix, and the last one states that one can choose local stochastic matrices to obtain any desired limiting distribution of their products.

\subsubsection{Convergence of products and invariance of limits} 
For a given $\cA\in \cS_G$, we define
\begin{equation}\label{eq:underlinea}
\underline {a}:= \min_{(v_i,v_j)\in E} (\min A_{ij}) \quad \mbox{and} \quad {\epsilon:= \underline{a}^{n - 1}.}
\end{equation}
The first main result is as follows: 

\vspace{.15in}

\begin{Theorem}\label{th:main1}
Let $G = (V, E)$ be a simple, connected undirected graph on $n$ nodes. 
Then, for every set of local stochastic matrices $A_{ij}\in \R^{n\times n}$,  $(v_i,v_j)\in E$, defined as in~\eqref{eq:deflsm}, the following two statements are equivalent:
\begin{enumerate}
    \item[(i)] There is a unique $p\in \rint \Delta^{n-1}$ such that for any infinite spanning sequence $\gamma$,  
$P_\gamma =\bfo p^\top.$
\vspace{.1cm}
    \item[(ii)] The $A_{ij}$ are holonomic for $G$. 
\end{enumerate}
Furthermore, if the $A_{ij}$ are holonomic for $G$ and $\gamma$ is $m$-spanning, then 
\begin{equation}\label{eq:exponentialconvergence} 
\|P_{\gamma}(t:0) \|_S \le (1-\epsilon)^{\frac{t}{m \lfloor \frac{n}{2} \rfloor} - 1}.
\end{equation}
\end{Theorem}

\vspace{.15in}

\begin{Remark}\normalfont
Note that uniqueness of the  probability vector $p$ is with respect to a given set of holonomic local stochastic matrices $A_{ij}$, and with respect to {\em all} infinite spanning sequences $\gamma$. 
The dependence of $p$ on the $A_{ij}$'s will be characterized shortly in Algorithm~1 below.
Also, it is easy to verify that the standard gossiping process has local stochastic matrices such that $\bar A_{ij} = \frac{1}{2} \bfo\bfo^\top$ for all $(v_i,v_j)\in E$, so $r_{ij} = 1$. Thus, these local stochastic matrices $A_{ij}$ are  holonomic for {\em any} connected graph.  For this special case, the corresponding probability vector~$p$ is simply $\frac{1}{n}\bfo$.
\end{Remark}

Note that an infinite spanning sequence can be obtained with probability one by selecting an edge out of $E$ uniformly at random. The following fact is then an immediate consequence of Theorem~\ref{th:main1}:

\vspace{.15in}

\begin{Corollary}\label{cor:randomsequence}
Let $\gamma$ be a simple random sequence obtained by selecting an edge out of $E$ uniformly at random.  
If the $A_{ij}$ are holonomic for $G$, then there exists a unique probability vector $p$ such that $P_\gamma = \bfo p^\top$ with probability one.   
\end{Corollary}

We characterize below the probability vector $p$. We do so by first presenting a positive vector, denoted by $q = [q_1;\cdots; q_n]$, and then normalizing it. The entire procedure is summarized in the following algorithm:

\noindent {\bf Algorithm 1.} Construction of $p$:  
\begin{description}
    \item[Step 1:] Pick an arbitrary node, say $v_1$, of $G$, and  set $q_1 := 1$.
    \item[Step 2:] For all  nodes $v_i$, $i \neq 1$, of $G$,  let $w$ be an arbitrary walk {\em from} $v_1$ {\em to} $v_i$ in $\vec G$ (since $G$ is connected, such a walk always exists). Define $q_i:= R_w$. 
    \item[Step 3:] Normalize the vector $q$ by 
    \begin{equation}\label{eq:defp}
        p := \frac{q}{\sum^n_{i = 1} q_i}. 
    \end{equation}
\end{description}

It should be clear that every entry of $q$, defined in Steps 1 and 2, is positive, so the vector $p$ is well defined. 

\vspace{.15in}

\begin{Theorem}\label{th:defp}
The probability vector $p$ in Theorem~\ref{th:main1} is given by~\eqref{eq:defp}.
\end{Theorem}

\vspace{.15in}

\begin{Remark}\normalfont
From its construction, $p$ appears to depend on both the base node chosen ($v_1$ above, Step~1), and on the walks from nodes $v_j$ to $v_1$ chosen (Step~2). 
On the way of proving the main results, we will show that, under the assumption that the local stochastic matrices are holonomic for $G$,  $p$ is in fact independent of these two parameters.
\end{Remark}

\subsubsection{The space of holonomic local stochastic matrices}

In this subsection, we study the set of $|E|$-tuples of holonomic local stochastic matrices for $G$ as a subset of $\cS_G$: 
\begin{equation}\label{eq:defHG}
\cH_G := \{\cA \in \cS_G \mid R_C=1 \mbox{ for all cycles in } \vec G\}.     
\end{equation}
Since  holonomic constraints arise only if cycles are present, if $G$ is a tree, then $\cH_G = \cS_G$.

By Theorems~\ref{th:main1} and~\ref{th:defp}, an element $\cA\in \cH_G$ gives rise to a unique probability vector $p$, defined in Algorithm~1.   
Formally, we define a map $\pi$ as follows: 
\begin{equation}\label{eq:defpi}
\pi: \cA \in \cH_G \mapsto p \in \rint \Delta^{n-1}. 
\end{equation}
Following the steps of Algorithm~1, it is easy to see that $\pi$ is analytic. We  now characterize the preimages $\pi^{-1}(p)$ precisely:

\vspace{.15in}

\begin{Theorem}\label{th:mainsurjective}
The map $\pi$ defined in~\eqref{eq:defpi} is surjective. For each $p\in \rint \Delta^{n-1}$, the preimage $\pi^{-1}(p)$ is an $|E|$-dimensional open box.  
\end{Theorem}

It is an immediate consequence of the theorem that the dimension of $\cH_G$ is $(n+|E|-1)$;  indeed, since the dimension of $\pi^{-1}(p)$ is independent of $p$ and since $\pi$ is onto $\Delta^{n-1}$, the dimension of $\cH_G$  is the sum of the dimension of $\Delta^{n-1}$, which is $(n-1)$, and the dimension of some (and, hence, any) preimage $\pi^{-1}(p)$, which is $|E|$. Note that the segments defining the box are not necessarily aligned with the coordinate axes, and will generally be slanted.

\section{Analysis and Proofs of Theorems}\label{sec:proofs}
In this section, we establish relevant propositions and  prove the main results. 

\subsection{Holonomy and Algorithm~1}\label{ssec:holoalgo}
In the subsection, we show that the output of Algorithm~1 
is indeed independent of the base node chosen in Step 1 and the walks chosen in Step 2. These statements are proven in Proposition~\ref{prop:pindepbase}, and in Proposition~\ref{prop:1forclosedwalk} and Corollary~\ref{cor:corpropclosedwalk} respectively.

\vspace{.15in}

\begin{Proposition}\label{prop:1forclosedwalk}
Let $A_{ij}$ be a set of local stochastic matrices that are holonomic for $G$, and $w$ be a closed walk in $\vec G$. Then, $R_w = 1$.
\end{Proposition}

\begin{proof}
Any closed walk $w$ can be decomposed, edge-wise, into a union of disjoint cycles, labeled as $C_1, \ldots, C_k$. Then, $R_w = R_{C_1}\cdots R_{C_k}$. From Definition~\ref{def:holonomicAij}, $R_{C_i} = 1$ for every $i =1,\ldots, k$ and, hence, $R_w = 1$.   
\end{proof}

The next result follows as a corollary to Proposition~\ref{prop:1forclosedwalk}:

\vspace{.15in}

\begin{Corollary}\label{cor:corpropclosedwalk}
Let $A_{ij}$ be a set of local stochastic matrices that are holonomic for $G$. Let $w$ and $w'$ be two distinct walks in $\vec G$ from the same node $v_i$ to the same node $v_j$. Then, $R_{w}=R_{w'}$.
\end{Corollary}

\begin{proof}
By concatenating $w$ with $w'^{-1}$, we obtain a closed walk, which we denote by $w^*$. On one hand, by Proposition~\ref{prop:1forclosedwalk}, $R_{w^*} = 1$. On the other hand, $R_{w^*} = R_w R_{w'^{-1}} = R_{w}/R_{w'}$. It follows that $R_{w} = R_{w'}$.  
\end{proof}

The above corollary has shown that if the base node $v_i$, chosen in Step~1 of Algorithm 1, is fixed, then the value of other entries $q_j$, for $j\neq i$, are independent of the choices of walks from $v_i$ to $v_j$. 
Though the value of the vector $q$ depends on a particular choice of   base node,  we show below that this dependence only changes a normalization constant. Consequently, the value of the vector $p$ in Step~3 is independent of said base node. 

We now let $q$ and $q'$ be the vectors obtained from Algorithm~1 by using $v$ and $v'$, respectively, as the base nodes. 

\vspace{.15in}

\begin{Proposition}\label{prop:pindepbase}
Let $w$ be any walk from $v'$ to $v$ in $\vec G$. Then, $q' = R_w q$.  
\end{Proposition}
  
\begin{proof}
We establish the proposition by showing that for $1\leq i\leq n$, we have $q'_i = R_w q_i$. Let $w_i$ be a walk from node $v$ to $v_i$ (if $v_i = v$, then $w_i$ can be the empty walk), and $w'_i$ be the walk from node $v'$ to $v_i$ obtained by concatenating the given $w$ and $w_i$. Then, $R_{w'_i} = R_w R_{w_i}$. 
From Algorithm 1, we know that $q_i = R_{w_i}$ and $q'_i = R_{w'_i}$, so $q'_i = R_w q_i$ as desired. 
\end{proof}

\subsection{A common left-eigenvector}\label{ssec:proppv}
This subsection is devoted to first showing that the vector $p$ of Theorem~\ref{th:defp} is in fact a left-eigenvector of all local stochastic matrices $A_{ij}$ for $G$, associated with the eigenvalue~$1$, provided that these matrices are holonomic for $G$. This result is the first step in the proof of Theorem~\ref{th:main1}. We also show here that if the holonomic constraint is not met, the vector $p$ is not well-defined.

\vspace{.15in}

\begin{Proposition}\label{prop:p}
There is a probability vector $p$ such that $p^\top A_{ij} = p^\top$ for all $(v_i,v_j)\in E$ if and only if the set of $A_{ij}$ is holonomic for $G$. Moreover, if such $p$ exists, then it is unique.  
\end{Proposition}

To prove Proposition~\ref{prop:p}, we use the following lemmas:

\vspace{.15in}

\begin{Lemma}\label{lem:uniquep}
Let $\bar r_{ij}:=[1; r_{ij}] \in \R^2$. 
Then, $\bar r_{ij}$ is the unique left-eigenvector (up to scaling) of $\bar A_{ij}$ corresponding to eigenvalue~$1$. 
\end{Lemma}

\begin{proof}
Recall that $r_{ij} = a_{ij}/a_{ji}$.  
It follows that  
\begin{align*}
 \bar A_{ij}^\top \bar r_{ij} & = \begin{bmatrix}
    1-a_{ij} & a_{ji} \\
    a_{ij} & 1-a_{ji}
    \end{bmatrix}
    \begin{bmatrix}
    1 \\
    r_{ij}
\end{bmatrix}
     = 
    \begin{bmatrix}
    (1-a_{ij}) + a_{ji} r_{ij} \\
    a_{ij}  + (1-a_{ji}) r_{ij}
    \end{bmatrix} \\
    & = 
    \begin{bmatrix}
    (1-a_{ij}) + a_{ij}   \\
    a_{ji} r_{ij} + (1-a_{ji}) r_{ij} 
    \end{bmatrix} 
    = 
    \begin{bmatrix}
    1 \\
    r_{ij}
    \end{bmatrix} = \bar r_{ij},
\end{align*}
so $\bar r_{ij}$ is a left-eigenvector of $\bar A_{ij}$ corresponding to eigenvalue~$1$. The uniqueness follows from the fact that $\bar A_{ij}$ is an irreducible stochastic matrix (with all entries being positive).    
\end{proof}

We next have the following fact: 

\vspace{.15in}

\begin{Lemma}\label{lem:newuniquep}
If there exists a vector $p \neq 0$ such that $p^\top A_{ij} = p^\top$ for any $(v_i,v_j)\in E$, then for any walk $w$ in $\vec G$ from node $v_\ell$ to node $v_k$, it holds that $p_k = R_w p_\ell$. 
\end{Lemma}

\begin{proof}
Since $p^\top A_{ij} = p^\top$, the vector $[p_i;p_j]$, with $i < j$, is a left-eigenvector of $\bar A_{ij}$ corresponding to eigenvalue $1$. 
From Lemma~\ref{lem:uniquep}, $[p_i;p_j]$ is necessarily proportional to $\bar r_{ij}$ and thus $p_j = r_{ij} p_i$. Thus, we can apply this relation repeatedly along the sequence of the edges in $w$ and obtain that $p_k = R_w p_\ell$.
\end{proof}

With Lemmas~\ref{lem:uniquep} and \ref{lem:newuniquep}
above, we prove Proposition~\ref{prop:p}:  

\begin{proof}[Proof of Proposition~\ref{prop:p}]
We first assume that the set of $A_{ij}$ is holonomic for $G$. 
Let $q$ and $p$ be given as in Algorithm 1. Since $q$ and $p$ differ by a multiplicative factor, it suffices to show that $q^\top A_{ij} = q^\top$. 
By construction (see Eq.~\eqref{eq:deflsm}), the matrix $A_{ij}$  is equal to the identity matrix save for the  $2\times 2$ principal submatrix at columns/rows $i$ and $j$, 
$\bar A_{ij}$. 
Without loss of generality, we assume that $i < j$, and let $\bar q_{ij}:= [q_i;q_j]$. It is enough to show that for any pair $i <j$,  
$\bar q_{ij}$ is a left-eigenvector of $\bar A_{ij}$ with  eigenvalue~$1$.   Since $(v_i,v_j)$ is an edge of $G$, $q_j = r_{ij} q_i$ by Algorithm~1. Thus, $\bar q_{ij}$ is proportional to $\bar r_{ij}$ introduced in Lemma~\ref{lem:uniquep} and, hence, $\bar q_{ij}^\top \bar A_{ij} = \bar q_{ij}^\top$.

We now show  that $p$ is the only probability vector that satisfies $p^\top A_{ij} = p^\top $ for all $(v_i,v_j)\in E$. 
Let $p'$ be another such probability vector. To every $A_{ij}$, the equality $p'^\top A_{ij} = p'^\top$ implies that the two entries $p'_i$ and $p'_j$ satisfy
$p'_j = r_{ij} p'_i$ by Lemma~\ref{lem:uniquep}.
By Lemma~\ref{lem:newuniquep}, if we fix a base node, say $v_1$, of $G$, and let $w$ be a walk from $v_1$ to $v_i$ (since $G$ is connected), then $p'_i = R_w p'_1$ for all $i = 2,\ldots,n$. From Step 2 in Algorithm~1, we see that $p'$ is proportional to $p$ and, hence, $p' = p$.

It remains to show that if the set of $A_{ij}$'s is not holonomic for $G$, then no  $p$ such that $p^\top A_{ij} = p$, for all $(v_i,v_j)\in E$, exists. 
We proceed  by contradiction and assume that there exists such a vector $p$.  Then, for any walk $w$ starting at $v_i$ and ending at $v_j$, we have from Lemma~\ref{lem:newuniquep} that $p_j= R_w p_i$. 

Because $R_w$ is always positive and because $\vec G$ is strongly connected, every entry of $p$ is nonzero (otherwise, $p$ has to be the zero vector). 
But, since the set of $A_{ij}$ is not holonomic, there exists a closed walk $w$ in $\vec G$ such that $R_{w}\neq 1$. Pick a node, say $v_i$, in $w$; then, $p_i = R_w p_i \neq p_i$, which is a contradiction. This completes the proof. 
\end{proof}

Let $\cA \in \cS_G$. To any spanning tree $G' = (V, E')$ of $G$, the corresponding subset of local stochastic matrices $A_{ij}$, for $(v_i,v_j) \in E'$, is always holonomic for $G'$. 

Thus, the follow result is an immediate consequence of Proposition~\ref{prop:p}:

\vspace{.15in}

\begin{Corollary}\label{cor:treetoprobvec}
To every spanning tree $G' = (V, E')$ of $G$ we can assign a unique probability vector $p'$ such that 
\begin{equation}\label{eq:propp}
p'^\top A_{ij} = p'^\top \mbox{ for all } (v_i, v_j)\in E'.
\end{equation}
\end{Corollary}

Note, in particular, that if the local stochastic matrices are holonomic for $G$, then for any two spanning trees $G'$ and $G''$ of $G$, their associated probability vectors $p'$ and $p''$ are equal. 
Conversely, we have the following:

\vspace{.15in}

\begin{Proposition}\label{prop:distinctpnonholo}
Suppose that the set of $A_{ij}$ is not holonomic for $G$; then, there exist two distinct spanning trees $G'$ and $G''$ of $G$ with distinct probability vectors~$p'$ and~$p''$ satisfying Eq.~\eqref{eq:propp} 
\end{Proposition}

\begin{proof}
Because the set of $A_{ij}$ is not holonomic for $G$, there exists at least one cycle $C = v_1v_2\cdots v_k v_1$ of $\vec G$ such that $R_C \neq 1$. 
Let $G'$ and $G''$ be two spanning trees such that $G'$ contains edges $(v_\ell, v_{\ell+1})$, for all $\ell = 1,\ldots,k-1$ and $G''$ contains the edge $(v_1,v_k)$. 
It should be clear that such $G'$ and $G''$ exist and are distinct because $G'$ cannot contain the edge $(v_1,v_k)$. 
 By Corollary~\ref{cor:treetoprobvec}, we can uniquely assign the
 probability vectors $p'$ and $p''$ to  $G'$ and $G''$, respectively. 
We claim that $p'$ and $p''$ are distinct. 
To establish the claim, we let $w := v_1\cdots v_k$ be the unique path in $\vec G'$ from $v_1$ to $v_k$. 
Then, since $p'$ satisfies Eq.~\eqref{eq:propp}, by Lemma~\ref{lem:newuniquep}, we have that
$p'_{k} = R_w p'_{1}$. 
Similarly, for $\vec G''$, since $v_1v_k$ is a directed edge, we have that  
$p''_{k} = r_{1k} p''_{1}$. 
But, $R_C = R_w r_{k1} = R_w /r_{1k}\neq 1$, which implies that the two ratios $p'_{k}/p'_{1}$ and $p''_{k}/p''_{1}$ are different. This completes the proof.
\end{proof}

\subsection{Uniform lower bound for nonzero entries of $P_\gamma$}

For a given vector $z\in \R^n_{\geq 0}$,  recall that $\supp(z)$ is the support of $z$. Let $\gamma$ be a walk in $G$ and $ z_\gamma:= P_\gamma z$. If $\gamma$ is an empty walk, then $P_\gamma=I$ and $z_\gamma = z$. 

We also recall that $\min z_{\gamma}$ the smallest non-zero entry of $z_\gamma$, i.e., $\min z_\gamma = \min \{ z_{\gamma,i} \mid i\in \supp(z_\gamma)\} $. 
Note that if $P$ is an arbitrary $n\times n$ nonnegative matrix with positive diagonal entries, then $\supp(z) \subseteq \supp(Pz)$; indeed, if $z_{i} > 0$, then  
$(Pz)_{i} \geq  P_{ii}z_{i}  > 0$. 
As a consequence, we have the following fact: 
\begin{equation}
\label{eq:suppinclusion}\gamma_1 \mbox{ is a string of } \gamma_2 \;\Rightarrow\; 
 \supp(z_{\gamma_1}) \subseteq \supp(z_{\gamma_2}).
\end{equation}

When we consider a nested family of edge strings $\gamma_1 \subseteq \gamma_2\subseteq \cdots $ for which the supports of the corresponding $z_{\gamma_i}$ are the same, the smallest non-zero entry over the support is non-decreasing as shown below:

\vspace{.15in}

\begin{Lemma}\label{lem:gammaveeedge}
Let $\gamma$ be a string of edges and $e$ be an edge of $G$. Let $\gamma':= \gamma \vee e$. If $\supp(z_{\gamma'}) = \supp(z_\gamma)$, then $\min z_{\gamma'}  \geq \min z_{\gamma}$. 
\end{Lemma}

\begin{proof}
For convenience, but without loss of generality, we assume that $e = (v_1, v_2)$, so $A_{12} =\diag(\bar A_{12}, I_{n-2})$, where $\bar A_{12}$ was defined in Eq.~\eqref{eq:defbarAij}.

The matrix $P_{\gamma'}$  differs from $P_{\gamma}$ in its first two rows only;  we denote by $z_{\gamma,1}$ and $z_{\gamma,2}$ the first two entries of $z_\gamma$. 
Since $\supp(z_{\gamma'}) = \supp(z_{\gamma})$ and since all entries of $\bar A_{12}$ are positive, 
we have that $z_{\gamma,1}$ is $0$ if and only if $z_{\gamma,2}$ is $0$; indeed, say $z_{\gamma,1} = 0$ and $z_{\gamma,2} > 0$, then, after multiplication, the first entry of $z_{\gamma'}$ will be positive, contradicting the fact that $\supp(z_{\gamma'}) = \supp(z_{\gamma})$. 
Now, consider the following two cases:  
\vspace{.1cm}

\noindent
{\em Case 1.} If $z_{\gamma,1} = z_{\gamma,2} = 0$, then $z_\gamma = z_{\gamma'}$ and, hence, $\min z_\gamma = \min z_{\gamma'}$. 
\vspace{.1cm}

\noindent
{\em Case 2.} If $z_{\gamma,1} \neq 0$ (and, hence, $z_{\gamma,2}\neq 0$), then the first and second entries of $z_{\gamma'}$ are given by  
$$
\begin{bmatrix}
z_{\gamma',1} \\
z_{\gamma',2}
\end{bmatrix} = 
\bar A_{12} 
\begin{bmatrix}
z_{\gamma,1} \\
z_{\gamma,2}
\end{bmatrix}.
$$
Because $\bar A_{12}$ is a stochastic matrix, both $z_{\gamma',1}$ and  $z_{\gamma',2}$ are convex combinations of $z_{\gamma,1}$ and $z_{\gamma,2}$.
Thus,
$$
\min 
\begin{bmatrix}
z_{\gamma',1} \\
z_{\gamma', 2}
\end{bmatrix} \ge \min \begin{bmatrix}
z_{\gamma,1} \\
z_{\gamma,2}
\end{bmatrix} \geq \min z_{\gamma}. 
$$
In either case, we conclude that $\min z_{\gamma'} \ge \min z_{\gamma}$.
\end{proof}

With Lemma~\ref{lem:gammaveeedge}, we can now establish a lower bound on  $\min P_\gamma$:  

\vspace{.15in}

\begin{Proposition}\label{prop:mindelta}
Let $\cA\in \cS_G$ and $\underline a$ be defined as~\eqref{eq:underlinea}.  Then, for any sequence $\gamma$ of edges,  
$$\min P_\gamma > \underline{a}^{n-1} = \epsilon.$$
\end{Proposition}

\begin{proof}
Let $\{e_i\}_{i = 1}^n$ be the standard basis of $\R^n$. Then, the $i$th column of $P_\gamma$, denoted by $P_{\gamma,i}$, is  given by $P_\gamma e_i$. It should be clear that $\min P_\gamma = \min_{i = 1}^n P_{\gamma,i}$. Thus, it suffices to show that $\min P_{\gamma,i} \geq \underline{a}^{n-1}$ for all $i = 1,\ldots, n$. 

To establish the fact, we fix an arbitrary $i\in \{1,\ldots, n\}$, and let $N_{\gamma, i}$ be the cardinality of $\supp P_{\gamma,i}$. We will show below that 
\begin{equation}\label{eq:inductiveproof}
\min P_{\gamma, i} \geq \underline{a}^{N_{\gamma,i} -1}.
\end{equation} 
Note that if~\eqref{eq:inductiveproof} holds, then the proof is done because $N_{\gamma,i}$ is bounded above by~$n$ and, hence, $\min P_{\gamma, i} \geq \underline{a}^{n-1}$. 

The proof of~\eqref{eq:inductiveproof} is by induction on $N_{\gamma,i}$. 
For the base case $N_{\gamma,i} = 1$, the sequence $\gamma$ can only comprise  edges $(v_j,v_k)$ that are not incident to node $v_i$. To see this,  note that by the definition of the local stochastic matrices, if $\gamma$ does not contain any edge incident to node $v_i$, then $P_{\gamma,i} = e_i$ and, hence, $N_{\gamma,i} = 1$.  
Next, we assume that $\gamma$ contains an edge incident to $v_i$, and let $\gamma_t=(v_i,v_j)$ be the first such edge in $\gamma$. 
Then, by the above arguments, $P_{\gamma,i}(t-1:0) = e_i$. Moreover, 
using the same arguments as in the proof of Lemma~\ref{lem:gammaveeedge}, we have that both the $i$th and $j$th entry of $P_{\gamma,i}(t:0)$ are nonzero. Further, since the support of $P_{\gamma,i}$ is monotonic by~\eqref{eq:suppinclusion}, we have $N_{\gamma,i}\geq 2$.   
We have thus shown that if $N_{\gamma,i} = 1$, then $P_{\gamma,i} = e_i$ and $\min P_{\gamma,i} = 1$.

For the inductive step, we assume that the statement holds for any $\gamma$ with $N_{\gamma,i} = k$ (for $1\leq k \leq n-1$), and prove that it holds for any $\gamma$ with $N_{\gamma,i} = k + 1$. 

For any given $\gamma$ with $N_{\gamma,i} = k + 1$, we let $t\ge 1$ be chosen such that the two strings $\gamma':=\gamma(t:0)$ and $\gamma'':= \gamma(t + 1:0)$ satisfy the condition that $N_{\gamma',i}= k$ and $N_{\gamma'',i} = k+1$. 
Such $t$ exists because $N_{\gamma(t:0),i}$ is a monotonically non-decreasing function in~$t$ due to Eq.~\eqref{eq:suppinclusion} and $\supp e_i = 1$.

Let $A_{ij}$ be the local stochastic matrix corresponding to the last edge in $\gamma''$. It is so that $P_{\gamma''} = A_{ij} P_{\gamma'}$. By the induction hypothesis, $\min P_{\gamma',i} \geq {\underline a}^{k-1}$, we have that
$$
\min P_{\gamma'',i} \geq \min A_{ij} \min P_{\gamma',i} \geq \underline{a}^{k}.  
$$
Finally, note that the sequence $\gamma$ is obtained from $\gamma''$ by adding edges to the end of $\gamma''$. One can thus iteratively apply Lemma~\ref{lem:gammaveeedge} to obtain that $\min P_{\gamma,i} \ge \min P_{\gamma'',i}$. This completes the proof. 
\end{proof}

\subsection{Proofs of Theorems~\ref{th:main1} and~\ref{th:defp}}

For a stochastic matrix $A\in\R^{n\times n}$, its coefficient of ergodicity \cite{seneta} is defined as
$$\mu(A) = \frac{1}{2}\max_{i,j}\sum_{k=1}^{n}|a_{ik}-a_{jk}|.$$
We always have that $\mu(A) \le 1$. 
It has been shown in \cite[Lemma 3]{hajnal} that for any two stochastic matrices $P$ and $Q$, 
\begin{equation}\label{eq:contraction}
    \|PQ\|_S \le \mu(P)\|Q\|_S.
\end{equation}

A stochastic matrix $A$ is called a {\em scrambling matrix} if no pair of  rows of $A$ are orthogonal.  
The following result is well known (see, e.g., Eq.~(25) in~\cite{reachingp2}):

\vspace{.15in}

\begin{Lemma}\label{lem:scrambling}
For any scrambling matrix $A$, 
\begin{equation*}\label{eq:scrambling}
    \mu(A)\le 1- \min A.
\end{equation*}
\end{Lemma}

Let $\gamma$ be a finite spanning sequence of edges of $G$. 
Then, by~\eqref{eq:suppinclusion}, the graph of $P_\gamma$ is strongly connected with self-arcs (more precisely, the graph contains a bi-directional spanning tree). 
It then follows that $P_\gamma$ is irreducible~\cite[Theorem 6.2.24]{horn1}.   
We also need the following lemma: 

\vspace{.15in}

\begin{Lemma}\label{lm:neighborshared}
The product of any set of $\ell\ge \lfloor \frac{n}{2} \rfloor$
irreducible $n\times n$ stochastic matrices with positive diagonal entries is a scrambling matrix.
\end{Lemma}

\begin{proof}
We will use a graphical approach. 
We call a digraph {\em neighbor-shared} if any two distinct nodes share a common in-neighbor.

Let $G_p$ and $G_q$ be two directed graphs with the same node set $V$. The composition  of  $G_p $ with $G_q$,  denoted by $G_q\circ G_p$, is a digraph with node set $V$ and edge set defined as follows:  $v_i v_j$ is an edge of $G_q\circ G_p$ whenever there is a node $v_k$
such that $v_iv_k$ is an edge of $G_{p}$ and $v_kv_j$ is an edge  of $G_{q}$. Since
composition is an associative binary operation, it extends unambiguously to any finite sequence of digraphs with the same node set.  
Let $M_1$ and $M_2$ be two nonnegative $n\times n$ matrices, and $G_1$, $G_2$ be their respective graphs. Then,  
by construction, the graph of $M_2M_1$ is $G_2\circ G_1$. 

From~\cite[Prop.~9]{reachingp1}, we have that the composition of any set of $\ell\ge \lfloor \frac{n}{2} \rfloor$ strongly connected graphs with self-arcs with the same node set is neighbor-shared. It has been shown in~\cite{reachingp2} that  a stochastic matrix is scrambling if and only if its graph is neighbor-shared. This concludes the proof.
\end{proof}

With the preliminaries above, we will now prove Theorems~\ref{th:main1} and~\ref{th:defp}. 

\begin{proof}[Proof of Theorems~\ref{th:main1} and~\ref{th:defp}]
We first assume that the $A_{ij}$ are holonomic for $G$ and prove the two theorems. 
Let $\gamma$ be an infinite spanning sequence. 
Since $\|P_{\gamma(t:0)}\|_S$ is monotonically non-increasing by~\eqref{eq:contraction}, the limit exists as $t$ goes to $\infty$. 
We show below that the limit is $0$. 
Let $0 =: t_0 < t_1 < t_2 \cdots $ be a monotonically increasing sequence such that every string $\gamma(t_{k+1}:t_{k})$, for $k \ge 0$, has $\lfloor \frac{n}{2} \rfloor$ disjoint spanning sub-strings. 
From Lemma~\ref{lm:neighborshared}, every product $P_{\gamma(t_{k+1}:t_{k})}$, for $k\ge 0$, is a scrambling matrix. 
By Proposition~\ref{prop:mindelta} and Lemma~\ref{lem:scrambling}, $\mu(P_\gamma(t_{k+1},t_k)) < (1 - \epsilon)$. It follows from the inequality~\eqref{eq:contraction} and $\|I\|_S=1$ that
\begin{align*}
    \lim_{k\to \infty} \|P_{\gamma(t_k:0)} \|_S 
    &\le \lim_{k\to\infty} (1 - \epsilon) \|P_{\gamma(t_{k-1}:0)} \|_S \\
    &\le \lim_{k\to\infty} (1 - \epsilon)^k = 0,  
\end{align*}
which implies that $\lim_{t\to \infty} \|P_{\gamma(t:0)}\|_S = 0$. 
It is known~\cite{chatterjee1977towards} that the semi-norm $\|P_{\gamma(t:0)}\|_S$ converges to $0$ if and only if $P_\gamma$ converges to a rank-one matrix. This establishes asymptotic convergence. 

If $\gamma$ is, furthermore, $m$-spanning, then the sequence $\{t_k\}_{k\ge 0}$ can be chosen such that $t_{k+1} - t_k \le m \lfloor \frac{n}{2}\rfloor =: T$. Let $t$ be an arbitrary time index and choose $k$ with $t_{k} \le t < t_{k+1}$. Then, 
\begin{align*}
    \|P_{\gamma(t:0)} \|_S &\le \|P_{\gamma(t_{k}:0)}\|_S \le (1 -\epsilon)^{k} = (1 - \epsilon)^{\lfloor \frac{t}{T} \rfloor} \\
    &\le (1 - \epsilon)^{\frac{t}{T} - 1},
\end{align*}
which establishes exponential convergence and Eq.~\eqref{eq:exponentialconvergence} in Theorem~\ref{th:main1}. 

To show that the vector $p$ is the one given in Algorithm 1, we first note that from Proposition~\ref{prop:p}, $p^\top A_{ij}= p^\top$ for all $(v_i, v_j)\in E$. Thus, $p^\top P_{\gamma(t:0)} = p^\top$ for all $t \ge 1$. Because $P_{\gamma(t:0)}$ converges to a rank one matrix as $t \to \infty$, it must converge to $\bfo p^\top$.

Finally, we assume that the $A_{ij}$ are not holonomic and show that there does not exist a probability vector $p$ such that $P_\gamma = \bfo p^\top$ for any infinite spanning sequence $\gamma$. Under the assumption on $A_{ij}$, owing to Proposition~\ref{prop:distinctpnonholo},  there exist at least two distinct spanning trees $G'$ and $G''$ of $G$ for which the associated probability vectors $p'$ and $p''$ are distinct. 
Let $\gamma'$ and $\gamma''$ be two infinite spanning sequences for $G$,  with the property that edges in $\gamma'$ (resp. $\gamma''$) belong to $G'$ (resp. $G''$).  
Because $G'$ and $G''$ are trees, the associated $(A_{ij})_{(v_i,v_j)\in E'}$ and $(A_{ij})_{(v_i,v_j)\in E''}$ are holonomic for $G'$ and $G''$, respectively. Thus, by the above arguments $P_{\gamma'} = \bfo p'$ and $P_{\gamma''} = \bfo p''$. Since $p'\neq p''$,  $P_{\gamma'} \neq P_{\gamma''}$. 
This completes the proof. 
\end{proof}

\subsection{Proof of Theorem~\ref{th:mainsurjective}}
Recall that  a local stochastic matrix $A_{ij}$ assigned to an undirected edge $(v_i,v_j)\in E$ gives rise to two ratios $r_{ij}=\frac{a_{ij}}{a_{ji}}$ and $r_{ji} = \frac{a_{ji}}{a_{ij}}$, which are inverse of each other, as defined in Section~\ref{ssec:lsmholo}. The set of all such ratios is thus the $|E|$-dimensional {\em subset} of $\R^{2|E|}$ defined as follows: 
$$ 
 \cY := \left \{(y_{ij})_{v_iv_j\in\vec E} \in \R_+^{2|E|} \mid y_{ij} y_{ji} = 1 \quad\forall v_iv_j \in \vec E \right \}.
$$
It is easy to see that $\cY$ is diffeomorphic to $\R^{|E|}_+$.

Also, recall that $\cS_G$ is the set of all $|E|$-tuples of local stochastic matrices for $G$.  
We now introduce the map $\phi: \cS_G  \to \cY$ defined as follows: 
\begin{equation}\label{eq:defphi}
\phi:\cS_G  \to \cY: (A_{ij})_{(v_i,v_j)\in E} \mapsto  \left(\frac{a_{ij}}{a_{ji}}\right)_{v_iv_j\in \vec E}. 
\end{equation}
Moreover, we have the following result:

\vspace{.15in}

\begin{Proposition}\label{prop:maphisujective}
The map $\phi$ defined in~\eqref{eq:defphi} is surjective and for any $y\in \cY$, the pre-image $\phi^{-1}(y)$ is an $|E|$-dimensional open box embedded in $\R^{2|E|}$.  
\end{Proposition}

\begin{proof}
The map $\phi$ can be realized as a {\em Cartesian product} of maps $\phi_{ij}: (0,1)\times(0,1) \to \R^2_+$, for $(v_i,v_j)\in E$ with $i < j$, where each $\phi_{ij}$ is defined by sending the matrix $A_{ij}$ to a pair of reciprocal ratios  $(a_{ij}/a_{ji}, a_{ji}/a_{ij})$, i.e., we have that
\begin{align*} 
\phi\left ((A_{ij})_{(v_i,v_j)\in E} \right )  & = \prod_{(v_i,v_j)\in E} \phi_{ij}(A_{ij}) \\
 &=\Big((a_{ij}/a_{ji}, a_{ji}/a_{ij})\Big)_{(v_i,v_j) \in E}.
\end{align*}
Thus, taking inverses, we obtain that
$$\phi^{-1}\left ((a_{ij}/a_{ji})_{v_iv_j\in \vec E} \right ) = \prod_{(v_i,v_j)\in E} \phi_{ij}^{-1}(a_{ij}/a_{ji}, a_{ji}/a_{ij}).$$
Now, let $(r_{ij})_{v_iv_j \in \vec E}$, with $r_{ij} > 0$ and $r_{ij}r_{ji} = 1$, be an arbitrary point in the codomain of $\phi$. 
We claim that  $\phi_{ij}^{-1}(r_{ij},r^{-1}_{ij})$ is nonempty and, moreover, it is an open bounded segment in $\R^2$. If the claim holds,  then the proof is complete: Indeed, 
if $\phi_{ij}^{-1}(r_{ij},r^{-1}_{ij})$ is nonempty, then $\phi_{ij}$ is surjective. Owing to the Cartesian product structure exhibited above, $\phi$ is also surjective. By the same arguments, if $\phi_{ij}^{-1}(r_{ij},r^{-1}_{ij})$ is an open bounded segment, then $\phi^{-1}((r_{ij})_{v_iv_j\in \vec E})$ is an open box.

We will now establish the claim stated above. For ease of presentation, we will represent the matrix $A_{ij}$ by the pair of entries $(a_{ij}, a_{ji})$ (recall that all the other entries of $A_{ij}$ are uniquely determined by this pair). This representation can be viewed as a bijective linear map. 
With this representation, 
it follows from computation that
$$
\phi^{-1}_{ij}(r_{ij},r^{-1}_{ij}) = 
\begin{cases}
 \{ (r_{ij} x, x  ) \mid 0 < x < 1\}   &  \mbox{if $r_{ij}\le 1$}, \\
\{ (x, r^{-1}_{ij} x ) \mid 0 < x < 1\}    &  \mbox{if $r_{ij} > 1$}.
\end{cases}
$$
Thus, the preimage is an open segment parameterized by $x\in (0,1)$ as is claimed.  
\end{proof}

The map $\phi$ relates the local stochastic matrices to the ratios $r_{ij}$, for $v_iv_j \in \vec E$. We next construct a map that relates these ratios to the probability vector $p$. 
To this end, let $\theta: \rint \Delta^{n-1} \to \R^{2|E|}_+$ defined as follows:
\begin{equation}\label{eq:deftheta}
    \theta: p  = [p_1 \;\cdots\; p_n]^\top \mapsto (p_j/p_i)_{v_iv_j \in \vec E}~.
\end{equation}

We will show that the map $\theta$ is one-to-one, and thus admits a well-defined inverse. To this end, we describe the image of $\theta$ explicitly, as an algebraic subset of $\R^{2|E|}$. 
For a given positive vector $y=(y_{ij})_{v_iv_j\in \vec E} \in \cY$ and for a given walk $w = v_1\cdots v_k$ in $\vec G$, we let $Y_w := \prod^{k-1}_{\ell = 1} y_{\ell,\ell + 1}$.  Define a subset of $\cY$ as follows:
\begin{equation}\label{eq:defY}
\mathcal{Y}_{\cH} := \left \{y\in \cY \mid Y_w = 1 \mbox{ for every closed walk $w$ of $\vec G$} \right \}.
\end{equation}
Note that if $\cA\in \cS_G$ is holonomic for $G$, then the corresponding vector of ratios $r = (r_{ij})_{v_iv_j\in \vec E}$ belongs to the set $\cY_{\cH}$ by construction.

We have the following result:

\vspace{.15in}

\begin{Proposition}\label{prop:onetooneontoY}
The map $\theta$ is one-to-one and onto $\cY_{\cH}$.    
\end{Proposition}

\begin{proof}
First,  if $y = \theta(p)$ for some $p\in \rint \Delta^{n-1}$, it follows from~\eqref{eq:deftheta} that $Y_w = 1$ for any closed walk $w = v_1\cdots v_k v_1$, so the image of $\theta$ is contained in $\cY_{\cH}$. 

Next, we show that the map $\theta$ is one-to-one. Let $p$ and $p'$ be two distinct vectors in $\rint \Delta^{n-1}$. Then, there exists at least a pair of distinct indices $(i,j)$ such that $p_j/p_i \neq p'_j / p'_i$. Indeed, if no such pair exists, then $p'$ is proportional to $p$ which, since  both $p$ and $p'$ belong to $\Delta^{n-1}$,  contradicts the fact that they are distinct. This shows that $\theta$ is one-to-one.

Finally, we show that 
for any $y\in \cY_{\cH}$, there exists a $p\in \rint \Delta^{n-1}$ such that $\theta(p) = y$. 
One can obtain such a vector $p$ by using Algorithm 1, but with $r_{ij}$ and $R_w$ replaced by $y_{ij}$ and $Y_w$, respectively. The choice of the base node and the choices of walks from the base node to the other nodes do not matter since $Y_w = 1$ for all closed walks $w$---the same arguments used in Propositions~\ref{prop:1forclosedwalk} and~\ref{prop:pindepbase}, and Corollary~\ref{cor:corpropclosedwalk} can be applied to establish the fact. Then, by construction, the vector $p$ indeed satisfies $\theta(p) = y$. To see this, we
let $v_iv_j$ be an arbitrary edge in $\vec G$ and show that $p_j/p_i = y_{ij}$. 
Let $v_i$ be a base node chosen in Step 1 of Algorithm 1. Since $v_iv_j$ is an edge, by Step 2 of Algorithm 1, we have that $p_j = y_{ij} p_i$, i.e., $p_j/p_i = y_{ij}$.
\end{proof}

With the propositions above, we prove Theorem~\ref{th:mainsurjective}: 

\begin{proof}[Proof of Theorem~\ref{th:mainsurjective}]
By Proposition~\ref{prop:onetooneontoY}, the map $\theta$ is a bijection. Moreover, by  Definition~\ref{def:holonomicAij} of holonomic local stochastic matrices, 
$\cH_G = \phi^{-1}(\cY_{\cH})$. 
We can thus write the map $\pi: \cH_G \to \rint \Delta^{n-1}$ as $\pi(\cdot) = \theta^{-1}(\phi(\cdot))$ by restricting the domain of $\phi$ to the subset $\cH_G$. 
Thus, for a given $p\in \rint \Delta^{n-1}$, since
$\pi^{-1}(p) = \phi^{-1}(\theta(p))$ and since $\theta(p)\in \cY_\cH \subset \cY$, we conclude from Proposition~\ref{prop:maphisujective} that $\pi^{-1}(p)$ is an $|E|$-dimensional open box.
\end{proof}

\section{Conclusions}\label{end}
In this paper, we have investigated convergence of {\em weighted} gossip processes and characterized their limits. Mathematically, a weighted gossip process can be expressed as an infinite product of local stochastic matrices, which are not required to be doubly stochastic.
Using the notion of holonomy, 
we have provided a necessary and sufficient condition for the product to converge to a unique rank-one matrix, independent of the order of the appearance of the stochastic matrices in the product. We characterized explicitly both the limit and the sets of holonomic stochastic matrices that can give rise to a desired limit.    
Amongst the future directions in which the present work can be extended, we mention generalization of the results to local stochastic matrices with zeros in the $2\times 2$ principal submatrices. This case, though seemingly close to the one studied here, in fact exhibits a very different asymptotic behavior. We will also aim to generalize the results to vector-valued gossip processes, and to establish a unified framework that accommodate the results of the paper and the results of the previous work~\cite{BC2020triangulated}.

\bibliographystyle{unsrt}
\bibliography{references}




\end{document}